\documentclass[a4paper,10pt]{article}
\usepackage{mathtext}
\usepackage[T1,T2A]{fontenc}
\usepackage[cp1251]{inputenc}
\usepackage[english]{babel}
\usepackage{amsmath}
\usepackage{amsfonts}
\usepackage{amssymb}
\usepackage{mathrsfs}
\usepackage{amsthm}
\usepackage{enumerate}
\usepackage{graphicx}
\graphicspath{}
\DeclareGraphicsExtensions{.pdf,.png,.jpg}

\usepackage{color}
\usepackage{euscript}

\usepackage{cite}

\newtheorem{Le}{Lemma}[section]
\newtheorem{Def}{Definition}[section]
\newtheorem{St}[Le]{Proposition}
\newtheorem{Th}{Theorem}[section]
\newtheorem{Cor}[Le]{Corollary}
\newtheorem{Rem}[Le]{Remark}

\numberwithin{equation}{section}

\DeclareMathOperator{\diam}{diam}

\title{Besicovitch's covering theorem in the parabolic metric}
\author{Nikita Dobronravov \footnote{Supported by the Ministry of Science and Higher Education of the Russian Federation (agreement no. 075-15-2025-343).}}
\begin{document}
	\maketitle
	\begin{abstract}
		We prove that Besicovitch's covering theorem holds for the parabolic metric  $d_p((x_1,t_1),(x_2,t_2))=(|x_1-x_2|^{p}+|t_1-t_2|^{\frac{p}{2}})^{\frac{1}{p}}$ on $\mathbb{R}^n\times\mathbb{R}$ if and only if $p\geqslant 2$. If $p=2$, this answers a question posed by P. Mattila in the affirmative. 
	\end{abstract}
	
	\section{Introduction}
	
	Covering theorems play an important role in mathematical analysis. A prominent example is Besicovitch's covering theorem.
	\begin{Th}[Besicovitch's covering theorem]
		Let $n\in\mathbb{N}$. Then there exist constants $M_1(n),M_2(n)\in\mathbb{N}$ with the following property. Let $A$ be a bounded subset of $\mathbb{R}^n$ and let $\mathcal{B}$ be a family of closed Euclidean balls such that each point of $A$ is the center of a ball of $\mathcal{B}$. Then there is a subfamily $\mathcal{B}'\subset\mathcal{B}$ such that 
		\begin{equation}\label{s-BCP}
			A\subset\bigcup\mathcal{B}',
		\end{equation}
		and 
		\begin{equation}\label{BCP}
			\sum\limits_{B_r(x)\in\mathcal{B}'}\chi_{B_r(x)}(y)
			\leqslant
			M_1(n)
		\end{equation}
		for every $y\in\mathbb{R}^n$.
		Moreover, the family $\mathcal{B}'$ can be split into $M_2(n)$ disjoint subfamilies.
	\end{Th}	
	Here, and in what follows, $B_r(x)$ is the closed ball with center $x$ and radius $r$.	By the notation $\cup\mathcal{B}$ we mean $\cup_{B\in\mathcal{B}}B$.
	
	There are many papers generalizing Besicovitch's theorem to the case of metric spaces other than Euclidean. 
	We recall the definitions of a Besicovitch family of balls and several versions of Besicovitch Covering Property for a metric space.
	\begin{Def}
		We say that a family $\mathcal{B}$ of balls in a metric space $(X, d)$ is a Besicovitch family of balls if, first, for every ball $B_r(x)\in\mathcal{B}$ we have $x\notin B_{\tau}(y)$ for all $B_{\tau}(y)\in\mathcal{B}$, $B_r(x)\neq B_{\tau}(y)$, and, second, $\cap_{B_r(x)\in\mathcal{B}}B_r(x)\neq\emptyset$. 
	\end{Def}
	\begin{Def}
		A metric space $(X,d)$ satisfies the strong Besicovitch Covering Property (s-BCP) if Besicovitch's covering theorem holds on it.
		
		A metric space $(X,d)$ satisfies the Besicovitch Covering Property (BCP) if a weaker version of the Besicovitch covering theorem holds, without the requirement that the family $\mathcal{B}'$ can be split into disjoint subfamilies.
		
		A metric space $(X,d)$ satisfies the weak Besicovitch Covering Property (w-BCP) if there exists a constant $M$ such that $\#\mathcal{B}\leqslant M$ for every Besicovitch family of balls.
	\end{Def}

	One of the most general results in this direction is the result of Le Donne and Rigot~\cite{DonRig2019}. They classified homogeneous groups that admit a homogeneous metric possessing BCP. They also proved (see~\cite[Theorem~1.6]{DonRig2017}) that BCP is not stable under a biLipschitz change of metric. Thus, the question of which homogeneous metrics satisfy BCP remains open.

	Let $1\leqslant p<\infty$ and let $d_p$ be the parabolic metric 
	\begin{equation}\label{pm}
		d_p((x_1,t_1),(x_2,t_2))
		=
		(|x_1-x_2|^p+|t_1-t_2|^{\frac{p}{2}})^{\frac{1}{p}}
	\end{equation}
	on $\mathbb{P}^n=\mathbb{R}^n\times\mathbb{R}$. Here, and in what follows, $|x|$ is the Euclidean norm of $x$. Note that there is a very similar homogeneous metric on the Heisenberg group $\mathbb{H}^n=\mathbb{C}^n\times\mathbb{R}$
	\begin{equation}
		d((z,t),(\omega,\tau))
		=
		(|z-\omega|^4+|t-\tau+2\operatorname{Im}(z_1\bar{\omega}_1+\cdots+z_n\bar{\omega}_n)|^2)^{\frac{1}{4}}
	\end{equation}
	which is often called the Koranyi metric. In the papers~\cite{KorRei} and~\cite{SawWhe} it was proved that the Koranyi metric on the Heisenberg group does not satisfy the BCP. The metric $d_4$ is also called the Koranyi metric on the parabolic space.
	In the paper~\cite{Mat} (see page~10, discussion before Lemma~4.7), P.~Mattila asks whether Besicovitch's theorem can be extended to the parabolic metric $d_2$. Itoh~\cite{Itoh2018} proved that $\mathbb{P}^n$ equipped with the metric $d_\infty((x_1,t_1),(x_2,t_2))=\max\{|x_1-x_2|,\sqrt{|t_1-t_2|}\}$ satisfies s-BCP.
	We provide an affirmative answer to Mattila's question:
	\begin{Th}\label{BP}
		Let $2\leqslant p<\infty$. Then $\mathbb{P}^n$ with the metric $d_p$ satisfies s-BCP.
	\end{Th}
	\begin{Rem}
		The constant in Theorem~\ref{BP} does not depend on $p$. The author does not have a limiting argument that allows one to deduce s-BCP for $d_\infty$ form Theorem 1.2. 
	\end{Rem}
	\begin{Th}\label{neBP}
		Let $1\leqslant p< 2$. Then $\mathbb{P}^n$ with the metric $d_p$ does not satisfy w-BCP.
	\end{Th}
	\begin{Rem}
		Theorem~\ref{neBP} holds for $0<p<1$ as well, but $d_p$ is only a quasi-metric in this case.
	\end{Rem}
	The proof of Theorem~\ref{BP} will be similar to the classical proof of Besicovitch's covering theorem in the Euclidean space. The only difference is that the proofs of Lemmas~\ref{>>},~\ref{><},~\ref{<>} below are based on the properties of the metric $d_p$.
	An analogue of these three lemmas for Euclidean space can be formulated as follows.
	\begin{St}
		Let the sequence of balls $\{B_{r_j}(x_j)\}_{j\in J}$ be such that $r_j>10$, $0\notin B_{r_j}(x_j)$, $B_1(0)\cap B_{r_j}(x_j)\neq\emptyset$ and $x_j\notin B_{r_i}(x_i)$ for $i\neq j$. Then $\langle x_i,x_j\rangle\leqslant \frac{9}{10}|x_i||x_j|$ for $i\neq j$.
	\end{St}
	Here and in what follows, $\langle x,y\rangle$ is the Euclidean scalar product of the vectors $x$ and $y$.
	
	To prove Theorem~\ref{neBP}, we construct an infinite sequence of balls such that $0$ belongs to all of them, and no ball in the sequence contains the center of any other ball.
	
	{\bf Acknowledgment.}
	I am grateful to my scientific adviser D. M. Stolyarov for statement of the problem and advice. I am also grateful to P. Mattila for attention to my work and inspiring questions.

	\section{New part of the proof}
	
	\begin{Def}
		For $p\geqslant 1$, let $B^p_r((x,t))$ be the closed ball in the metric $d_p$ on $\mathbb{P}^n$, centred at $(x,t)$ with radius $r$.
	\end{Def}
	
	\begin{St}\label{nerav}
		Let $(x,t)\in \mathbb{P}^n$. Then
		\begin{equation}
			d_p(B_1^p((0,0)),(x,t))^p
			\geqslant
			\max\{|x|-1,0\}^p+\max\{|t|-1,0\}^{\frac{p}{2}}.
		\end{equation}
	\end{St}
	\begin{proof}
		Let $(a,b)\in B_1^p((0,0))$ be such that $d_p(B_1^p((0,0)),(x,t))=d_p((a,b),(x,t))$. Then
		\begin{equation}
			d_p(B_1^p((0,0)),(x,t))^p
			=
			|x-a|^p+|t-b|^{\frac{p}{2}}
			\geqslant
			\max\{|x|-1,0\}^p+\max\{|t|-1,0\}^{\frac{p}{2}}.
		\end{equation}
	\end{proof}
	\begin{St}\label{stn}
		Let $q\geqslant 1$ and $0\leqslant b\leqslant a$. Then 
		\begin{equation}
			a^{q}-b^q
			\geqslant
			a^{q-1}(a-b).
		\end{equation}
	\end{St}
	\begin{proof}
		We can write
		\begin{equation}
			a^{q}-b^q
			=
			a^q\Big(1-\Big(1-\frac{a-b}{a}\Big)^q\Big)
			\geqslant
			a^q\Big(1-\Big(1-\frac{a-b}{a}\Big)\Big)
			=
			a^{q-1}(a-b).
		\end{equation}
	\end{proof}
	\begin{Le}\label{>>}
		Let $p\geqslant 2$. Let also $\mathcal{A}=\{B^p_{r_j}((x_j,t_j))\}$ be a sequence of balls in $\mathbb{P}^n$ that satisfies the following requirements:
		\begin{enumerate}
			\item For every $j$, $t_j\geqslant 10$, $|x_j|\geqslant10$, and $B^p_{r_j}((x_j,t_j))\cap B^p_{1}((0,0))\neq\emptyset$.
			
			\item If $i\neq j$, then $(x_j,t_j)\notin B_{r_i}^p((x_i,t_i))$.
			
			\item For every $i$ and $j$ the inequality $\langle x_j,x_i\rangle\geqslant\frac{99}{100}|x_i||x_j|$ is true.
		\end{enumerate}
		Then $\mathcal{A}$ has length at most $8$.
	\end{Le}
	\begin{proof}
		Assume the contrary, let $\#\mathcal{A}\geqslant 9$. We may assume that $|x_j|\geqslant|x_{j+1}|$ for all $j$.
		
		First, we use 
		$B^p_{r_j}((x_j,t_j))\cap B^p_{1}((0,0))\neq\emptyset$: 
		\begin{equation}\label{B10dp}
			r_j^p
			\geqslant 
			d_p(B^p_1((0,0)),(x_j,t_j))^p
			\overset{\scriptscriptstyle{\text{Prop.~\ref{nerav}}}}{\geqslant}
			(|x_j|-1)^p+(t_j-1)^{\frac{p}{2}}.
		\end{equation}
		Since $(x_{j+1},t_{j+1})\notin B^p_{r_j}((x_j,t_j))$, we have
		\begin{equation}\label{main}
			\begin{aligned}
				r_j^p
				&\leqslant
				d_p((x_j,t_j),(x_{j+1},t_{j+1}))^p
				=
				|x_j-x_{j+1}|^p +|t_{j}-t_{j+1}|^{\frac{p}{2}}\\
				&=
				(|x_j|^2+|x_{j+1}|^2-2\langle x_j,x_{j+1}\rangle)^{\frac{p}{2}}+|t_{j}-t_{j+1}|^{\frac{p}{2}}\\
				&\leqslant
				\Big(|x_j|^2+|x_{j+1}|^2-\frac{198}{100}|x_j||x_{j+1}|\Big)^{\frac{p}{2}}+|t_{j}-t_{j+1}|^{\frac{p}{2}}.
			\end{aligned}
		\end{equation}
		Hence,
		\begin{equation}\label{pn}
			(|x_j|-1)^p+(t_j-1)^{\frac{p}{2}}
			\leqslant
			\Big(|x_j|^2+|x_{j+1}|^2-\frac{198}{100}|x_j||x_{j+1}|\Big)^{\frac{p}{2}}+|t_{j}-t_{j+1}|^{\frac{p}{2}}.
		\end{equation}
		We want to use Proposition~\ref{stn} for $q=\frac{p}{2}$, $a=(|x_j|-1)^2$, and $b=|x_j|^2+|x_{j+1}|^2-\frac{198}{100}|x_j||x_{j+1}|$.
		We need to check that $b\leqslant a$:
		\begin{equation}
			\begin{aligned}
				a-b
				&=
				\frac{198}{100}|x_j||x_{j+1}|-|x_{j+1}|^2-2|x_j|+1\\
				&=
				|x_j||x_{j+1}|\Big(\frac{198}{100}-\frac{|x_{j+1}|}{|x_j|}-\frac{2}{|x_{j+1}|}+\frac{1}{|x_j||x_{j+1}|}\Big)
				\geqslant
				|x_j||x_{j+1}|\Big(\frac{198}{100}-1-\frac{1}{5}\Big)
				>
				0.
			\end{aligned}
		\end{equation}
		So,
		\begin{equation}\label{pny}
			\begin{aligned}
				&(|x_j|-1)^p-\Big(|x_j|^2+|x_{j+1}|^2-\frac{198}{100}|x_j||x_{j+1}|\Big)^{\frac{p}{2}}\\
				&\!\!\!\!\!\!\!\!\overset{\scriptscriptstyle{\text{Prop.~\ref{stn}}}}{\geqslant}
				(|x_j|-1)^{p-2}\Big(\frac{198}{100}|x_j||x_{j+1}|-|x_{j+1}|^2-2|x_j|+1\Big)\\
				&\geqslant
				(|x_j|-1)^{p-2}\Big(\frac{98}{100}|x_j||x_{j+1}|-2|x_j|+1\Big)
				\geqslant
				\frac{7}{10}(|x_j|-1)^{p-2}|x_j||x_{j+1}|\\
				&\geqslant
				\frac{7}{10} \Big(\frac{9}{10}\Big)^{p-2}|x_j|^{p-1}|x_{j+1}|.
			\end{aligned}
		\end{equation}
		Combining inequalities~\eqref{pn} and~\eqref{pny}, we get, in particular, that
		\begin{equation}\label{pn2}
			0
			<
			|t_{j+1}-t_{j}|^{\frac{p}{2}}-(t_j-1)^{\frac{p}{2}}.
		\end{equation}
		Therefore,
		\begin{equation}\label{2t+1}
			t_{j+1}
			\geqslant
			2t_j-1
			\geqslant
			\frac{3}{2}t_j.
		\end{equation}
		We return to inequalities~\eqref{pn} and~\eqref{pny}, which for $j=1$ imply
		\begin{equation}
			\frac{7}{10}\cdot \Big(\frac{9}{10}\Big)^{p-2}|x_1|^{p-1}|x_{2}|
			\leqslant
			|t_{2}-t_{1}|^{\frac{p}{2}}-(t_1-1)^{\frac{p}{2}}
			<
			t_2^{\frac{p}{2}}.
		\end{equation}
		Thus,
		\begin{equation}
			t_2^{\frac{1}{2}}
			\geqslant
			\Big(\frac{7}{10}\Big)^{\frac{1}{p}}\cdot \Big(\frac{9}{10}\Big)^{\frac{p-2}{p}}|x_1|^{\frac{p-1}{p}}|x_2|^{\frac{1}{p}}
			\geqslant
			\frac{8}{10}|x_2|.
		\end{equation}
		Hence,
		\begin{equation}\label{3/2}
			t_8^{\frac{1}{2}}
			\overset{\scriptscriptstyle{\eqref{2t+1}}}{\geqslant}
			\Big(\frac{3}{2}\Big)^3t_2^{\frac{1}{2}}
			\geqslant
			\Big(\frac{3}{2}\Big)^3\cdot\frac{8}{10}|x_2|
			\geqslant
			\Big(\frac{3}{2}\Big)^2|x_8|.
		\end{equation}
		Since $(x_{8},t_8)\notin B^p_{r_9}((x_9,t_9))$, we have
		\begin{equation}
			(|x_9|-1)^{p}+(t_9-1)^{\frac{p}{2}}
			\overset{\scriptscriptstyle{\text{Prop.~\ref{nerav}}}}{\leqslant}
			d_p((x_8,t_8),(x_9,t_9))^p
			=
			|x_9-x_8|^p+(t_9-t_8)^{\frac{p}{2}}.
		\end{equation}
		So,
		\begin{equation}
			(t_9-1)^{\frac{p}{2}}-(t_9-t_8)^{\frac{p}{2}}
			\leqslant
			|x_9-x_8|^p
			\leqslant
			2^p|x_8|^p
			\overset{\scriptscriptstyle{\eqref{3/2}}}{\leqslant}
			2^p\Big(\frac{2}{3}\Big)^{2p}t_8^{\frac{p}{2}}
			=
			\Big(\frac{8}{9}\Big)^{p}t_8^{\frac{p}{2}}
			\leqslant
			\Big(\frac{8}{9}\Big)^{\frac{p}{2}}t_8^{\frac{p}{2}}.
		\end{equation}
		Since $\frac{p}{2}\geqslant 1$ and $t_8>1$, the function $(\cdot-1)^{\frac{p}{2}}-(\cdot-t_8)^{\frac{p}{2}}$ is non-decreasing on $[t_8,\infty)$. 
		Hence by inequality~\eqref{2t+1}
		\begin{equation}
			\Big(\frac{3}{2}t_8-1\Big)^{\frac{p}{2}}-\Big(\frac{3}{2}t_8-t_8\Big)^{\frac{p}{2}}
			\leqslant
			(t_9-1)^{\frac{p}{2}}-(t_9-t_8)^{\frac{p}{2}}
			\leqslant
			\Big(\frac{8}{9}\Big)^{\frac{p}{2}}t_8^{\frac{p}{2}}.
		\end{equation}
		Therefore,
		\begin{equation}
			\Big(\frac{3}{2}-\frac{1}{t_8}\Big)^{\frac{p}{2}}-\Big(\frac{1}{2}\Big)^{\frac{p}{2}}
			\leqslant
			\Big(\frac{8}{9}\Big)^{\frac{p}{2}}.
		\end{equation}
		Thus,
		\begin{equation}
			\Big(\frac{3}{2}-\frac{1}{10}\Big)^{\frac{p}{2}}
			\leqslant
			\Big(\frac{1}{2}\Big)^{\frac{p}{2}}+\Big(\frac{8}{9}\Big)^{\frac{p}{2}}
			\leqslant
			\Big(\frac{1}{2}+\frac{8}{9}\Big)^{\frac{p}{2}}.
		\end{equation}
		This is a contradiction.
	\end{proof}
	\begin{Le}\label{><}
		Let $p\geqslant 2$. Let also $\mathcal{A}=\{B^p_{r_j}((x_j,t_j))\}$ be a sequence of balls in $\mathbb{P}^n$ that satisfies the following requirements:
		\begin{enumerate}
			\item For every $j$, $t_j\geqslant 500$, $|x_j|\leqslant10$, and $B^p_{r_j}((x_j,t_j))\cap B^p_{1}((0,0))\neq\emptyset$.
			
			\item If $i\neq j$, then $(x_j,t_j)\notin B_{r_i}^p((x_i,t_i))$.
		\end{enumerate}
		Then $\mathcal{A}$ has length at most $1$.
	\end{Le}
	\begin{proof}
		Assume the contrary, let $\#\mathcal{A}\geqslant 2$. We also may assume that $t_1\geqslant t_2$. Since $(x_2,t_2)\notin B_{r_1}^p((x_1,t_1))$, we have
		\begin{equation}
			(t_1-1)^{\frac{p}{2}}
			\overset{\scriptscriptstyle{\text{Prop.~\ref{nerav}}}}{\leqslant}
			r_1^p
			\leqslant
			(t_1-t_2)^{\frac{p}{2}}+|x_1-x_2|^p.
		\end{equation}
		Therefore,
		\begin{equation}
			499^{\frac{p}{2}}
			\leqslant
			(t_1-1)^{\frac{p}{2}-1}(t_2-1)
			\overset{\scriptscriptstyle{\text{Prop.~\ref{stn}}}}{\leqslant}
			(t_1-1)^{\frac{p}{2}}-(t_1-t_2)^{\frac{p}{2}}
			\leqslant
			|x_1-x_2|^p
			\leqslant
			20^p.
		\end{equation}
		This is a contradiction.
	\end{proof}
	\begin{Le}\label{<>}
		Let $p\geqslant 2$. Let also $\mathcal{A}=\{B^p_{r_j}((x_j,t_j))\}$ be a sequence of balls in $\mathbb{P}^n$ that satisfies the following requirements:
		\begin{enumerate}
			\item For every $j$, $|t_j|\leqslant 10$, $|x_j|\geqslant10$, and $B^p_{r_j}((x_j,t_j))\cap B^p_{1}((0,0))\neq\emptyset$.
			
			\item If $i\neq j$, then $(x_j,t_j)\notin B_{r_i}^p((x_i,t_i))$.
			
			\item For every $i$ and $j$ the inequality $\langle x_j,x_i\rangle\geqslant\frac{99}{100}|x_i||x_j|$ is true.
		\end{enumerate}
		Then $\mathcal{A}$ has length at most $1$. 
	\end{Le}
	\begin{proof}
		Assume the contrary, let $\#\mathcal{A}\geqslant 2$. We also may assume that $|x_1|\geqslant |x_2|$. Since $(x_2,t_2)\notin B_{r_1}^p((x_1,t_1))$, we have
		\begin{equation}
			(|x_1|-1)^p
			\overset{\scriptscriptstyle{\text{Prop.~\ref{nerav}}}}{\leqslant}
			r_1^p
			\leqslant
			|t_1-t_2|^{\frac{p}{2}}+|x_1-x_2|^p.
		\end{equation}
		Therefore,
		\begin{equation}
			\begin{aligned}
				9^{p-2}\cdot70
				&\leqslant
				(|x_1|-1)^{p-2}\Big(\frac{198}{100}|x_1||x_2|-|x_2|^2-2|x_1|+1\Big)\\
				&\overset{\scriptscriptstyle{\text{Prop.~\ref{stn}}}}{\leqslant}
				(|x_1|-1)^p-|x_1-x_2|^p
				\leqslant
				|t_1-t_2|^{\frac{p}{2}}
				\leqslant 20^{\frac{p}{2}}.
			\end{aligned}
		\end{equation}
		Hence,
		\begin{equation}
			\frac{70}{81}
			\leqslant
			\Big(\frac{20}{81}\Big)^{\frac{p}{2}}
			\leqslant 
			\frac{20}{81}.
		\end{equation}
		This is a contradiction.
	\end{proof}
	
	\section{Standard part of the proof}
	\begin{Def}
		Let $q>0$ and $p\geqslant 2$. The sequence of balls $\{B^p_{r_j}((x_j,t_j))\}$ in $\mathbb{P}^n$ is called a $q$-Besicovitch sequence if it satisfies the following condition: If $j<i$, then $r_i\leqslant(1+q)r_j$ and $(x_i,t_i)\notin B^p_{r_j}((x_j,t_j))$.
	\end{Def}
	\begin{Le}\label{stand}
		Let $n\in\mathbb{N}$, $q>0$, and $p\geqslant 2$. There exists a constant $C(n,q)\in\mathbb{N}$ with the following property. For every $q$-Besicovitch sequence $\mathcal{A}=\{B^p_{r_j}((x_j,t_j))\}$ and every $j$ the inequality
		\begin{equation}
			\#\big\{i\mid i<j\ \text{and }\ B^p_{r_i}((x_i,t_i))\cap B^p_{r_j}((x_j,t_j))\neq\emptyset\big\}\leqslant C(n,q)
		\end{equation}
		is true.
	\end{Le}
	\begin{Cor}\label{end}
		A $q$-Besicovitch sequence can be divided into $C(n,q)+1$ disjoint subsequences.
	\end{Cor}
	To prove Lemma~\ref{stand}, we need some additional statements.
	\begin{St}\label{distq}
		Let $q>0$, $p\geqslant 2$, and $\mathcal{A}=\{B^p_{r_j}((x_j,t_j))\}$ be a $q$-Besicovitch sequence. Then $\{B^p_{\frac{r_j}{3+q}}((x_j,t_j))\}$ is disjoint. 
	\end{St}
	\begin{proof}
		Let $j<i$. The condition $(x_i,t_i)\notin B^p_{r_j}((x_j,t_j))$ implies
		\begin{equation}
			d_p((x_j,t_j),(x_i,t_i))
			\geqslant
			r_j.
		\end{equation}
		Therefore,
		\begin{equation}
			d_p((x_j,t_j),(x_i,t_i))
			\geqslant
			r_j
			\geqslant
			\frac{r_j+r_i}{2+q}
			>
			\frac{r_j+r_i}{3+q}.
		\end{equation}
	\end{proof}
	\begin{Cor}\label{s0}
		Let $q>0$ and let $\mathcal{A}=\{B^p_{r_j}((x_j,t_j))\}_{j=1}^{\infty}$ be a $q$-Besicovitch sequence. Let $A\subset\mathbb{P}^n$ be a bounded set and $(x_j,t_j)\in A$ for all $j$. Then
		\begin{equation}
			\lim_{j\rightarrow\infty}r_j=0.
		\end{equation}
	\end{Cor}
	\begin{proof}
		The balls $B^p_{\frac{r_j}{3+q}}((x_j,t_j))$ are disjoint and $B^p_{\frac{r_j}{3+q}}((x_j,t_j))\subset B^p_{2r_1+\diam(A)}((x_1,t_1))$. Therefore, $\inf r_j=0$, so $\lim r_j=0$.
	\end{proof}
	\begin{St}\label{coes}
		Assume that $n\in\mathbb{N}$, $q>0$, and $M>1$. Then there exists a constant $K(n,q,M)\in\mathbb{N}$ with the following property. For every $q$-Besicovitch sequence $\mathcal{A}=\{B^p_{r_j}((x_j,t_j))\}$, $j\in\mathbb{N}$ and $l\in\mathbb{Z}$, the inequality
		\begin{equation}\label{car}
			\#\big\{i\mid i<j,\ M^lr_j\leqslant r_i<M^{l+1}r_j\ \text{and }\ B^p_{r_i}((x_i,t_i))\cap B^p_{r_j}((x_j,t_j))\neq\emptyset\big\}\leqslant K(n,q,M)
		\end{equation}
		is true.
	\end{St}
	\begin{proof}
		Let $\mathcal{C}_l$ be the set whose cardinality we wish to bound in~\eqref{car}. The inequality $r_j\leqslant (1+q)r_i$ implies that for $l<-\big\lceil\frac{\log(1+q)}{\log M}\big\rceil-1$ we have $\mathcal{C}_l=\emptyset$. Let $l\geqslant-\big\lceil\frac{\log(1+q)}{\log M}\big\rceil-1$. Here and in what follows, the notation $\lceil x\rceil$ means the ceiling of $x$.
		 If $i\in\mathcal{C}_l$, then
		\begin{equation}
			B^p_{\frac{r_i}{3+q}}((x_i,t_i))\subset B^p_{2(1+M^{l+1})r_j}((x_j,t_j)).
		\end{equation} 
		By Proposition~\ref{distq} 
		\begin{equation}
			\begin{aligned}
				(2(1+M^{l+1}))^{n+2}\mathcal{H}_{n+2}(B^p_{r_j}((x_j,t_j)))
				&=
				\mathcal{H}_{n+2}(B^p_{2(1+M^{l+1})r_j}((x_j,t_j)))\\
				&\geqslant
				\sum_{i\in\mathcal{C}_l} \mathcal{H}_{n+2}(B^p_{\frac{r_i}{3+q}}((x_i,t_i)))\\
				&\geqslant 
				(\#\mathcal{C}_l)M^{l(n+2)}(3+q)^{-(n+2)}\mathcal{H}_{n+2}(B^p_{r_j}((x_j,t_j))).
			\end{aligned}
		\end{equation}
		Here $\mathcal{H}_{n+2}$ denotes the Hausdorff measure of dimension $n+2$ related to the metric $d_p$. Notice that the Hausdorff dimension of $\mathbb{P}^n$ is $n+2$.
		So,
		\begin{equation}
			\begin{aligned}
				\#C_l
				&\leqslant
				\frac{(2(1+M^{l+1}))^{n+2}(3+q)^{n+2}}{M^{l(n+2)}}
				=
				2^{n+2}(3+q)^{n+2}\Big(M^{-l}+M\Big)^{n+2}\\
				&\leqslant
				2^{n+2}(3+q)^{n+2}\Big(M^{\lceil\frac{\log(1+q)}{\log M}\rceil+1}+M\Big)^{n+2}.
			\end{aligned}
		\end{equation}
	\end{proof}
	
	\begin{proof}[Proof of Lemma~\ref{stand}]
		By homogeneity of $\mathbb{P}^n$, we may assume that $B^p_{r_j}((x_j,t_j))=B^p_1((0,0))$. There exists a constant $F(n)\in\mathbb{N}$ such that  
		\begin{equation}
			\mathcal{A}=\bigcup_{m=1}^{F(n)}\mathcal{A}_m
		\end{equation}
		and for every $m$ and all $B^p_{r_{i_1}}((x_{i_1},t_{i_1})),B^p_{r_{i_2}}((x_{i_2},t_{i_2}))\in\mathcal{A}_m$, we have $\langle x_{i_1},x_{i_2}\rangle\geqslant(1-\frac{1}{100})|x_{i_1}||x_{i_2}|$.
		Let $M=\max(1000, 2+q)$, $I_m=\{i\in\mathbb{N}\mid B^p_{r_{i}}((x_{i},t_{i}))\in\mathcal{A}_m\}$, and
		\begin{equation}
			\mathcal{J}=\{i\in\mathbb{N}\mid\ i<j,\ \ r_i\leqslant M\ \text{and }\ B^p_{r_i}((x_i,t_i))\cap B^p_{1}((0,0))\neq\emptyset\},
		\end{equation}
		\begin{equation}
			\mathcal{C}_{l,m,+}=\{i\in I_m\mid i<j,\ t_i\geqslant 10,\ |x_i|\geqslant 10,\ M^l\leqslant r_i<M^{l+1}\ \text{and }\ B^p_{r_i}((x_i,t_i))\cap B^p_{1}((0,0))\neq\emptyset\},
		\end{equation}
		\begin{equation}
			\mathcal{C}_{l,m,-}=\{i\in I_m\mid i<j,\ t_i\leqslant -10,\ |x_i|\geqslant 10,\ M^l\leqslant r_i<M^{l+1}\ \text{and }\ B^p_{r_i}((x_i,t_i))\cap B^p_{1}((0,0))\neq\emptyset\},
		\end{equation}
		\begin{equation}
			\mathcal{K}_{l,m}=\{i\in I_m\mid i<j,\ |t_i|\leqslant 10,\ M^l\leqslant r_i<M^{l+1}\ \text{and }\ B^p_{r_i}((x_i,t_i))\cap B^p_{1}((0,0))\neq\emptyset\},
		\end{equation}
		\begin{equation}
			\mathcal{P}_{l,+}=\{i\in\mathbb{N}\mid i<j,\ t_i\geqslant 500,\ |x_i|\leqslant 10,\ M^l\leqslant r_i<M^{l+1}\ \text{and }\ B^p_{r_i}((x_i,t_i))\cap B^p_{1}((0,0))\neq\emptyset\},
		\end{equation}
		\begin{equation}
			\mathcal{P}_{l,-}=\{i\in\mathbb{N}\mid i<j,\ t_i\leqslant -500,\ |x_i|\leqslant 10,\ M^l\leqslant r_i<M^{l+1}\ \text{and }\ B^p_{r_i}((x_i,t_i))\cap B^p_{1}((0,0))\neq\emptyset\}.
		\end{equation}
		By Proposition~\ref{coes}, $\#\mathcal{J}\leqslant 2K(n,q,M)$.
		Let $\mathcal{D}=\{B^p_{r_{i_h}}((x_{i_h},t_{i_h}))\}$ be a sequence such that for every $h$ we have
		$i_h\in \mathcal{C}_{l_h,m,+}$, $l_h\geqslant 1$, and if $h_1\neq h_2$, then $|l_{h_1}-l_{h_2}|\geqslant 2$. By Lemma~\ref{>>} we have $\#\mathcal{D}\leqslant 8$.
		Therefore, for any $m$, we have
		\begin{equation}
			\#\big\{l\geqslant 1\mid \mathcal{C}_{l,m,+}\neq\emptyset\big\}\leqslant 16.
		\end{equation}
		Similarly,
		\begin{equation}
			\#\big\{l\geqslant 1\mid \mathcal{C}_{l,m,-}\neq\emptyset\big\}\leqslant 16.
		\end{equation}
		Also similarly by Lemma~\ref{<>}
		\begin{equation}
			\#\big\{l\geqslant 1\mid \mathcal{K}_{l,m}\neq\emptyset\big\}\leqslant 2.
		\end{equation}
		And by Lemma~\ref{><}
		\begin{equation}
			\#\big\{l\geqslant 1\mid \mathcal{P}_{l,+}\neq\emptyset\big\}\leqslant 2,
		\end{equation}
		\begin{equation}
			\#\big\{l\geqslant 1\mid \mathcal{P}_{l,-}\neq\emptyset\big\}\leqslant 2.
		\end{equation}
		It follows that
		\begin{equation}
			C(n,q)\leqslant K(n,q,\max(1000, 2+q))(34F(n)+6).
		\end{equation}
	\end{proof}
	\subsection*{Proof of Theorem~\ref{BP}.}
	\
	
	Case 1: $\sup\{r_j|\ B^p_{r_j}((x_j,t_j))\in \mathcal{B},\,(x_j,t_j)\in A\}=\infty$. Then there exists $j$ such that $A\subset B^p_{r_j}((x_j,t_j))$.
	
	Case 2: $\sup\{r_j|\ B^p_{r_j}((x_j,t_j))\in \mathcal{B},\,(x_j,t_j)\in A\}<\infty$. Fix some $q>0$. We will inductively construct the families $\mathcal{B}^j$ and balls $B^p_{r_{i_j}}((x_{i_j},t_{i_j}))$.
	Let $\mathcal{B}^0=\{B^p_r((x,t))\in\mathcal{B}|\ (x,t)\in A\}$. Suppose we have constructed $\mathcal{B}^{j}$.
		
	{\bf Case $\mathcal{B}^{j}\neq\emptyset$}. Let $B^p_{r_{i_j}}((x_{i_j},t_{i_j}))\in\mathcal{B}^j$ be such that 
	\begin{equation}
		r_{i_j}
		\geqslant
		\frac{1}{1+q}\sup\{r_i\mid B^p_{r_{i}}((x_{i},t_{i}))\in\mathcal{B}^j\},
	\end{equation}
	and define
	\begin{equation}
		\mathcal{B}^{j+1}
		=
		\{B^p_{r_{i}}((x_{i},t_{i}))\in\mathcal{B}^j\mid (x_{i},t_{i})\notin B^p_{r_{i_j}}((x_{i_j},t_{i_j}))\}.
	\end{equation}
	
	{\bf Case $\mathcal{B}^{j}=\emptyset$.} We end the process. 
	
	One may see that $\{B^p_{r_{i_j}}((x_{i_j},t_{i_j}))\}$ is a $q$-Besicovitch sequence. By Corollary~\ref{s0}, $\lim\limits_{j\to\infty} r_{i_j}=0$, or $\{B^p_{r_{i_j}}((x_{i_j},t_{i_j}))\}$ is a finite sequence. Therefore,
	\begin{equation}
		A\subset \bigcup_{j}B^p_{r_{i_j}}((x_{i_j},t_{i_j})).
	\end{equation}
	Corollary~\ref{end} completes the proof of Theorem~\ref{BP}.
	
	\section{Proof of Theorem~\ref{neBP}}
	We will prove that for $1\leqslant p<2$ there exists an infinite sequence of balls $\{B^p_{r_j}((x_j,t_j))\}$ such that $(0,0)\in B^p_{r_j}((x_j,t_j))$ and if $i\neq j$, then $(x_i,t_i)\notin B^p_{r_j}((x_j,t_j))$.
	
	Let $\{t_j\}_{j=1}^{\infty}$ be a rapidly growing infinite sequence such that $t_1\geqslant 1$, $t_j<t_{j+1}$, and
	\begin{equation}\label{ys}
		\Big(\Big(\frac{p}{2}(t_j+1)^{\frac{p}{2}-1}\Big)^{\frac{1}{p}}-\Big(\frac{p}{2}(t_{j+1}+1)^{\frac{p}{2}-1}\Big)^{\frac{1}{p}}\Big)^p
		\geqslant
		\frac{t_j+1}{(t_{j+1}+1)^{\frac{2-p}{2}}}.
	\end{equation}
	Let $r_j=(t_j+1)^{\frac{1}{2}}$ and $|x_j|^p=\frac{p}{2}(t_j+1)^{\frac{p}{2}-1}$.
	First, we prove that $(0,0)\in B^p_{r_j}((x_j,t_j))$:
	\begin{equation}
		\begin{aligned}
			d_p((0,0),(x_j,t_j))^p
			&=
			|x_j|^p+t_j^{\frac{p}{2}}
			=
			t_j^{\frac{p}{2}}+\frac{p}{2}(t_j+1)^{\frac{p}{2}-1}\\
			&=
			(t_j+1)^{\frac{p}{2}}\Big(\Big(1-\frac{1}{t_j+1}\Big)^{\frac{p}{2}}+\frac{p}{2}(t_j+1)^{-1}\Big)
			\leqslant
			(t_j+1)^{\frac{p}{2}}
			=r_j^p.
		\end{aligned}
	\end{equation}
	Second, we prove that if $j\neq i$, then $(x_i,t_i)\notin B^p_{r_j}((x_j,t_j))$. It suffices to consider the case $i<j$, because $r_i<r_j$.
	\begin{equation}
		\begin{aligned}
			d_p((x_i,t_i),(x_j,t_j))^p
			&=
			|x_i-x_j|^p+(t_j-t_i)^{\frac{p}{2}}
			=
			|x_i-x_j|^p+(t_j+1)^{\frac{p}{2}}\Big(1-\frac{t_i+1}{t_j+1}\Big)^{\frac{p}{2}}\\
			&>
			|x_i-x_j|^p+(t_j+1)^{\frac{p}{2}}\Big(1-\frac{t_i+1}{t_j+1}\Big)
			=
			r_j^p+|x_i-x_j|^p-\frac{t_i+1}{(t_j+1)^{\frac{2-p}{2}}}
			\overset{\scriptscriptstyle{\eqref{ys}}}{\geqslant}
			r_j^p.
		\end{aligned}
	\end{equation}
%
	
	\
	
	\
	
	\
	
	{\small
		
		St. Petersburg State University Department of Mathematics and Computer Sciences.\\
		e-mail: n.dobronravov@spbu.ru\\
		\bigskip
		
	}
	
\end{document}